\documentclass[]{article}

\usepackage{amsfonts}
\usepackage{amsthm}
\usepackage{amscd} 
\usepackage{amsmath}
\usepackage{mathtools}
\usepackage[top=30truemm,bottom=30truemm,left=30truemm,right=30truemm]{geometry}
\usepackage{indentfirst}
\newtheorem{thm}{Theorem}[section]
\newtheorem*{thm*}{Theorem A}
\newtheorem{lem}[thm]{Lemma}
\newtheorem{prop}[thm]{Proposition}
\newtheorem*{prop*}{Proposition}

\theoremstyle{definition}

\newtheorem{dfn}[thm]{Definition}
\newtheorem{exa}[thm]{Example}
\newtheorem{rmk}[thm]{Remark}

\newcommand{\relmiddle}[1]{\mathrel{}\middle#1\mathrel{}}

\title{The Canonical Lattice Isomorphism between \\
Topologies Compatible with a Linear Space
 and
 Subspaces}
\author{Takanobu Aoyama\footnote{t-aoyama@cr.math.sci.osaka-u.ac.jp}}
\begin{document}

\maketitle

\begin{abstract}
We consider all compatible topologies of an arbitrary finite-dimensional vector space over a non-trivial valuation field whose metric completion is a locally compact space. We construct the canonical lattice isomorphism between the lattice of all compatible topologies on the vector space and the lattice of all subspaces of the vector space whose coefficient field is extended to the complete valuation field. 
Moreover, in this situation, we use this isomorphism to characterize the continuity of linear maps between finite-dimensional vector spaces endowed with given compatible topologies, and also, we characterize all Hausdorff compatible topologies. 
\end{abstract}

\section{Introduction}
The notion of topological vector space is a generalization of the normed spaces. This space is a vector space endowed with a compatible topology (see Definition \ref{topological vector space} for the word ``compatible''). 
In this paper, we consider the set $\tau_K(X)$ of all compatible topologies on a fixed finite-dimensional vector space $X$ over a fixed Hausdorff topological field $K$.
In \cite{Bir}, G. Birkhkoff studied the set $\Sigma(X)$ of all topologies on $X$ with the inclusion order $\subset$ and showed that  
the partially ordered set $(\Sigma(X),\subset)$ has an algebraic structure, namely a lattice structure (see Definition \ref{lattice} for definition of ``lattice'').  
Then, $\tau_K(X)$ is a sub-partially ordered set of $\Sigma(X)$ and also has a lattice structure (moreover, complete lattice structure) with the inclusion order. \par
For a given Hausdorff topological field $K$ and a vector space $X$ over $K$, our goal is to construct a lattice isomorphism between the lattice $\tau_K(X)$ and a lattice of all linear subspaces of a vector space. By this isomorphism, we can handle compatible topologies easier because we can calculate a basis of corresponding subspaces. For example, we can calculate lattice operations of compatible topologies and we can determine which compatible topologies are Hausdorff topologies.   
Y. Chen studied in \cite{Che}, geometric aspects of open neighborhoods in non-Hausdorff compatible topologies, introducing a notion of strip-space.
Also, he stated that when coefficient field is $\mathbb{R}$ or $\mathbb{C}$, non-Hausdorff compatible topologies are determined by the closure of zero, which is a linear subspace.
Indeed, it is valid, when a coefficient field $K$ is a non-trivial complete valuation field that all compatible topologies $\tau_K(X)$ correspond to all its subspaces $\sigma_K(X)$ by a map $F$ which sends a compatible topology $T$ to a closure of zero ${\rm cl}_T[\{0\}]$ with respect to the topology $T$. We denote its inverse map by $G:\sigma_K(X)\rightarrow\tau_K(X)$.    
However, when the coefficient field of $X$ is a {\bf non-trivial, non-complete valuation field} such as $\mathbb{Q}$, the above correspondence does not hold generally, because more than one compatible Hausdorff topology is mapped to the same zero-dimensional subspace, in other words, the map $F$ is not injective.
In this paper, we recover the above correspondence when $X$ is a finite-dimensional vector space over a non-trivial, non-complete valuation field whose metric completion is locally compact. 
It is known that a non-trivial locally compact valuation field is a finite extension of $\mathbb{R}$, $\mathbb{Q}_p$ or $\mathbb{F}_p((t))$ (See \cite{Pon}). Thus, more specifically, we consider a finite product space $X$ of subfield $K$ of them.   
In this case, we conduct a metric completion of $K$ denoted by $\hat{K}$ and a scalar extension defined by $\hat{X}\coloneqq \hat{K} \bigotimes_K X$.   
Then, we construct a more general correspondence $\hat{F}$ and $\hat{G}$ between all compatible topologies $\tau_K(X)$ and all subspaces of $\hat{X}$ denoted by $\sigma_{\hat{K}}(\hat{X})$. More concretely, $\hat{F}$ sends every compatible topology $T$ to the intersections of closures of images of zero's open neighborhoods by the inclusion map $I:X \rightarrow \hat{X}$, where we take closures with respect to the strongest compatible topology $T_{\hat{K}}^{\max}(\hat{X})$ on $\hat{X}$. Namely, the map $\hat{F}$ is defined by:
$$
\hat{F}(T)\coloneqq \bigcap_{0\in U\in T}{\rm cl}_{T_{\hat{K}}^{\max}(\hat{X})}[I(U)].
$$
This map $\hat{F}$ is an analogy of $F$ because $F(T)={\rm cl}_T[\{0\}]$ is equal to the intersections of all open neighborhoods of zero.
For the map $\hat{G}$, because $\hat{X}$ is a finite-dimensional vector space over the non-trivial complete valuation field $\hat{K}$, every subspace $S$ corresponds to compatible topology $G(S)$ on $\hat{X}\supset I(X)$ and then, $\hat{G}(S)$ is defined as a relative topology of $G(S)$, where we identify $X$ and its image $I(X)$. 
Thus, our main theorem (\S 3, Theorem \ref{main thm}) is stated as follows:
\begin{thm*}
Let $K$ be a non-trivial, (possibly non-complete) valuation field whose metric completion is locally compact and $X$ be a finite-dimensional vector space over $K$.
Then, the maps $\hat{F}:\tau_K(X) \rightarrow \sigma_{\hat{K}}(\hat{X})$ and $\hat{G}:\sigma_{\hat{K}}(\hat{X}) \rightarrow \tau_K(X)$ are inverse maps \rm{(}see \S3, Definition \ref{FS} for definitions of $\hat{F}$ and $\hat{G}$ \rm{)}. 
Moreover, $\hat{F}$ and $\hat{G}$ invert the inclusion relation.
\end{thm*}
As a result, $\tau_K(X)$ is a lattice, which is isomorphic to all subspaces $\sigma_{\hat{K}}(\hat{X})$ with the inverted inclusion relation.\par
As an application, by this correspondence, we show two propositions which relate the topological notions to the corresponding subspaces.
First, we give an equivalent condition for a linear map being continuous with respect to given compatible topologies on the domain and the codomain (\S 5, Proposition \ref{continuous}).  More precisely, let $K$ be the same valuation field in the above Theorem A. and $X,Y$ be two finite-dimensional topological vector spaces over $K$. We denote given compatible topologies on $X$ and $Y$ by $T_X$ and $T_Y$, respectively. Then, for an arbitrary $K$-linear map $L:X \rightarrow Y$ is continuous if and only if the image of the corresponding subspace $\hat{F}(T_X)$ by a map $\hat{L}:\hat{X}\rightarrow\hat{Y}$ is contained in $\hat{F}(T_Y)$, where the map $\hat{L}$ is a $\hat{K}$-linear map such that the following diagram commutes:
\[
  \begin{CD}
     \hat{K}\times X @>{{\rm id}_{\hat{K}} \times L}>> \hat{K} \times Y \\
  @V{\otimes}VV    @V{\otimes}VV \\
     \hat{X}   @>{\hat{L}}>>  \hat{Y}.
  \end{CD}
\]
Second, we show that given compatible topology $T$ is Hausdorff if and only if the intersection of the corresponding subspace $\hat{F}(T)$ with $I(X)$ is $\{0\}$.
\par
The outline of this paper is as follows: In Section 2, we introduce and prepare some notations and propositions, and we recall a correspondence between compatible topologies and subspaces when the coefficient field is a non-trivial {\bf complete} valuation field. In Section 3, we consider a vector space over a non-trivial {\bf non-complete} valuation field.
By a metric completion of the valuation field, we construct a correspondence between its all compatible topologies and all linear subspaces of a vector space over the complete valuation field. The Section 4 is devoted to prove the main theorem (Theorem \ref{main thm})  and in Section 5, we show two propositions as applications of the main theorem.
%%%%%%%%%%%%%%%%%%%%%%%%%%%%%%%%%%%%%%%%%%%%%%%%%%%%%%%%%%%%%%%%%%%%%%%%%%%%%%%%%%%%%%%%%%%%%%%%%%%%%%%%%%%%%%%%%%%%%%%%%%%%%%%%%%%%%%%%%%%%%%%%%%%%%%%%%%%%%%%%%%%%%%%%%%%%%%%%%%%%%%section 2
\section{Preliminaries}
\begin{center}
{\bf In this paper, the term ``topology'' on a set $X$ means a family of subsets of $X$ which satisfies the open set axioms.}\\
\end{center}

\subsection{Lattices}
\begin{dfn}\label{lattice}
A partially ordered set $(L,\leq)$ is called a {\it lattice} if, for arbitrary two elements $x,y$ of $L$, the set $\{x,y\}$ has a supremum $s$ and an infimum $i$ in $L$, where {\it supremum} ({\it infimum}, resp.) is an element $s$ ($i$, resp.) of $L$ satisfying the following two properties:
\begin{enumerate}
\item $x \leq s$ and $y \leq s$.\, ($i \leq x$ and $i \leq y$, resp.)
\item $x \leq s',~ y \leq s'$ implies $s \leq s'$ for all element $s'$ of $L$.\, ($i' \leq x,~ i' \leq y$ implies $i' \leq i$ for all element $i'$ of $L$, resp.) 
\end{enumerate}
A lattice $(L,\leq)$ is called {\it complete} if every subset of $L$ has its supremum and infimum.\\  
Because every two elements $x,y$ of a lattice have a unique supremum and a unique infimum defined above, thus, we define binary operators $x \vee y$ and $x \wedge y$ called the {\it join} and the {\it meet}, respectively, by taking the supremum and the infimum of $\{x,y\}$.    
\end{dfn}

\begin{dfn}
Let $(L_1,\leq_1)$ and $(L_2, \leq_2)$ be lattices. A map $f:L_1 \rightarrow L_2$ is called a {\it lattice homomorphism} if $f$ preserves their joins and meets, that is, $f$ satisfies $f(x \vee_1 y) = f(x) \vee_2 f(y)$ and $f(x \wedge_1 y)= f(x) \wedge_2 f(y)$ for all 
$x,y \in L_1$, where $\vee_1,\vee_2$ are the joins of $L_1,L_2$ and $\wedge_1,\wedge_2$ are the meets of $L_1,L_2$, respectively. A bijective lattice homomorphism is called a {\it lattice isomorphism}. Two lattices are called {\it isomorphic} if there is a lattice isomorphism between them.
\end{dfn}
Note that a bijection between two lattices such that preserve their orders is a lattice isomorphism.

\begin{exa} 
The following are two examples of lattices.
\begin{enumerate}
\item Let $X$ be a set, and let $\Sigma(X)$ denote the set of all topologies on $X$, namely elements of $\Sigma(X)$ are families of subsets of $X$ which satisfy the open set axioms. 
A partially ordered set $(\Sigma(X), \subset)$ is a lattice. For two topologies $T_1,T_2$ on $X$, the join of $\{T_1,T_2\}$ is the topology whose subbase is $T_1 \cup T_2$ and the meet of $\{T_1,T_2\}$ is $T_1 \cap T_2$.
\item Let $X$ be a linear vector space, and let $\sigma(X)$ denote the set of all linear subspaces of $X$. Then, $(\sigma(X), \subset)$ is a lattice, where the join of $\{S_1,S_2\}$ is $S_1 +S_2$ and the meet of $\{S_1,S_2\}$ is $S_1 \cap S_2$.
\end{enumerate}
\end{exa}

\subsection{Topological Vector Spaces}
\begin{dfn}\label{topological field}
A {\it topological field $K$} is a commutative field endowed with a Hausdorff topology with which the following three operators are continuous:
\begin{itemize}
\item the addition: $K \times K \ni (\alpha, \beta) \mapsto \alpha + \beta \in K$,
\item the multiplication: $K \times K \ni (\alpha, \beta) \mapsto \alpha \times \beta \in K$, 
\item the multiplicative inverse: $K\setminus \{0\} \ni \alpha \mapsto \alpha^{-1} \in K\setminus \{0\}$,
\end{itemize}
where we endow $K \times K$ and $K\setminus \{0\}$ with the product topology and the relative topology of $K$, respectively.
\end{dfn}

\begin{dfn}\label{topological vector space}
Let $K$ be a topological field and $X$ be a vector space over $K$. We call a (possibly non-Hausdorff) topology $T$ on $X$ is {\it compatible with $X$} if  the following two operators of $X$ are continuous with respect to $T$:
\begin{itemize}
\item the addition: $X \times X \ni (x,y) \mapsto x+y \in X$, 
\item the scalar multiplication: $K \times X \ni (\alpha,x) \mapsto \alpha \cdot x \in X$, 
\end{itemize}
where we endow $X \times X$ and $K \times X$ with product topologies. A pair $(X,T)$ is called a {\it topological vector space}.
\end{dfn}

\begin{rmk}
If we consider a topological field $K$ which is not a Hausdorff space, it is known that $K$ is an indiscrete  topological space \rm{(}see \cite{Bou} for a proof\rm{)}. Then, a vector space $X$ over $K$ can have only one compatible topology with $X$, namely the indiscrete topology (that is, $\tau_K(X) =\{\{\phi,X\}\}$). Therefore, we assume that a topological field $K$ is a Hausdorff space for the rest of this paper. 
\end{rmk}

\begin{dfn}\label{Minkowski sum}
Let $X$ be a vector space over a field $K$. For subsets $A,B$ of $X$ and for a subset $L$ of $K$, we define the subsets $A+B$ and $L\cdot A$ of $X$ by
\begin{align}
A+B \coloneqq& \{\,a+b \in X \mid a \in A, b\in B\,\}, \nonumber \\
L\cdot A \coloneqq& \{\, \alpha \cdot a \in X \mid \alpha \in L, a \in A\,\}. \nonumber 
\end{align}
For simplicity, we denote $\{a\}+B$ and $A+\{b\}$ by $a +B$ and $A+b$, respectively and also denote $\{\alpha\} \cdot A$ by $\alpha \cdot A$.
\end{dfn}

In this paper, we denote by ${\rm cl}_T[A]$, the closure of the set $A$ with respect to the topology $T$ and by ${\rm int}_T[A]$, the interior of the set $A$.
These notations are abbreviated to ${\rm cl}[A]$ and ${\rm int}[A]$, respectively when there is no danger of confusion.\par 
Because the shift map $X \ni x \mapsto x+a \in X$ for every fixed element $a \in X$ is a self-homeomorphism of a topological vector space $(X,T)$, we have the following:
\begin{enumerate}
\item $a + {\rm cl}[A] = {\rm cl}[a+A]$ for all $a \in X$ and all subsets $A$ of $X$.\\
\item ${\rm cl}[A] +{\rm cl}[B] \subset {\rm cl}[A+B]$ for all subsets $A,B$ of $X$.\\
\item $A +U = \bigcup_{a \in A}(a+ U) \in T$ for all subsets $A$ and for all open subsets $U \in T$.
\end{enumerate}
Here, we give a proof of 2. For an element $a$ in $A$, we have $a+ {\rm cl}[B] = {\rm cl}[a+ B] \subset {\rm cl}[A+B]$ from 1. Thus, $A +{\rm cl}[B]$ is contained in ${\rm cl}[A+B]$.
By taking the closure of both sides of $A + b \subset {\rm cl}[A+B]$ for an element $b$ in ${\rm cl}[B]$, we obtain that ${\rm cl}[A] +b = {\rm cl}[A+b] \subset {\rm cl}[A+B]$ . Therefore, we have ${\rm cl}[A] +{\rm cl}[B] \subset {\rm cl}[A+B]$. 

\begin{dfn}
Let $K$ be a topological field, and let $X$ be a vector space over $K$. We denote, by $\tau_K(X)$, the set of all compatible topologies on $X$.
Also, we define a subset $\tau_K^H(X)$ of $\tau_K(X)$ consisting of all Hausdorff compatible topologies on $X$.
We abbreviate $\tau_K(X)$, $\tau_K^H(X)$ to $\tau(X)$, $\tau^H(X)$, respectively if the coefficient field $K$ is clear.\par
We denote the set of all $K$-linear subspaces of $X$ by $\sigma_K(X)$. We also, abbreviate $\sigma_K(X)$ to $\sigma(X)$ if there is no danger of confusion. 
\end{dfn}
\begin{center}
{\bf Note that we do not identify homeomorphic topologies in this paper.}
\end{center}

\begin{dfn}
Let $X$ and $Y$ be vector spaces over a topological field $K$, and let $f:X \rightarrow Y$ be a linear map.
\begin{enumerate}
\item For  $T \in \tau(Y)$, the {\it topology induced on $X$ by $f$} is
$$
\{\,f^{-1}(V)\mid V \in T\,\} \eqqcolon f^*(T).
$$
\item For  $T' \in \tau(X)$, the {\it topology coinduced on $Y$ by $f$} is
$$
\{\,V \subset Y\mid f^{-1}(V) \in T'\, \} \eqqcolon f_*(T').
$$ 
\end{enumerate}
\end{dfn}

We can see easily that $f^*$ is a map from $ \tau(Y)$ to $ \tau(X)$ and that, if $f$ is surjective, $f_*$ is a map from $\tau(X)$ to $\tau(Y)$.\\

Next, we see that $(\tau_K(X),\subset)$ is a complete lattice and has the maximum element. Consider an arbitrary family $\{T_{\lambda}\}_{\lambda \in \Lambda}$ of $\tau_K(X)$. The topology $T_{\sup}$ generated by $\bigcup_{\lambda \in \Lambda} T_{\lambda}$ belongs to $\tau_K(X)$ because the addition: $(X,T_{\sup})\times (X,T_{\sup}) \rightarrow (X,T_{\lambda})$ and the scalar multiplication: $K \times (X,T_{\sup}) \rightarrow (X,T_{\lambda})$ are continuous for all $\lambda \in \Lambda$. 
This implies that the partially ordered set $(\tau_K(X),\subset)$ is a complete lattice because the above argument directly implies that every subset $\{T_{\lambda}\}_{\lambda \in \Lambda}$ of $\tau_K(X)$ has the supremum and we can obtain its infimum by taking the supremum of the lower bound of $\{T_{\lambda}\}_{\lambda \in \Lambda}$. 
Therefore $(\tau_K(X), \subset)$ is a complete lattice and in particular, has the maximum (strongest) element. We introduce a notation to denote this topology.
\begin{dfn}\label{max topology}
Let $X$ be a vector space over a topological field $K$. Then, $T_K^{\max}(X)$ denotes the maximum topology in the $\tau_K(X)$ with respect to the inclusion. We abbreviate $T^{\max}_K(X)$ to $T^{\max}(X)$ or to $T^{\max}$ if there is no danger of confusion.
\end{dfn}

\begin{prop}\label{Hausdorff}
$(X,T_K^{\max}(X))$ is a Hausdorff space, that is,  $T_K^{\max}(X) \in \tau_K^H(X)$.
\end{prop}

\begin{proof}
Note that a topological vector space $(X,T)$ is a Hausdorff space if and only if, for all non-zero element $x$ in $X$, there is an open neighborhood of zero which does not contain $x$. We denote the topology of the topological field $K$ by $T_K$.\par
Let $x_0$ be a non-zero element of $X$. Then, $X$ is decomposed into a direct sum of $\mathop{\mathrm{span}}_K\{x_0\}$ and a linear subspace $S$, where $\mathop{\mathrm{span}}_K\{x_0\}$ is a linear subspace generated by $\{x_0\}$. Let $L$ be a linear map from $X$ to $K$ defined by
$$
X \ni x = \alpha \cdot x_0 + s \mapsto \alpha \in K,
$$
where $\alpha \cdot x_0$ and $s$ are components of $x$ with respect to the direct sum. 
Since $K$ is a Hausdorff space, we can take disjoint open subsets $V_1$ and $V_2$ in $K$ separating $0$ and $1$. Then, $L^{-1}(V_1)$ is an open neighborhood of zero with respect to $L^*(T_K)$ that does not contain $x_0$. Because $L^*(T_K)$ is in $\tau(X)$, and by the definition of $T^{\max}(X)$, we have $L^*(T_K) \subset T^{\max}(X)$, which implies $L^{-1}(V_1)$ belongs to $T^{\max}(X)$.
Therefore, there is an open neighborhood of zero that does not contain $x_0$ in $(X,T^{\max}(X))$ and we conclude that $(X,T^{\max}(X))$ is a Hausdorff space.   
\end{proof}
Next, we define maps which will be used to construct a correspondence between $\tau_K(X)$ and $\sigma_K(X)$.
\begin{dfn}\label{FG}
Let $X$ be a vector space over a topological field $K$. We define maps $F:\tau_K(X) \rightarrow \sigma_K(X)$ and $G_0:\sigma_K(X) \times \tau_K^H(X) \rightarrow \tau_K(X)$ by
\begin{align}
F(T) \coloneqq& \bigcap_{0 \in U \in T}U, \nonumber \\
G_0(S,T) \coloneqq& \pi_S^* \circ {\pi_S}_* (T), \nonumber
\end{align}
where $\pi_S: X \rightarrow X/S$ is the quotient map. 
\end{dfn}

In the next proposition, we show that $F(T)$ is actually a subspace of $X$ and that $F(T)$ is equal to the closure of zero, ${\rm cl}_T[\{0\}]$.

\begin{prop}\label{g is subspace}
For every $T$ in $\tau_K(X)$, the subset $F(T)$ of $X$ is a $K$-subspace of $X$ and is equal to ${\rm cl}_T[\{0\}]$.
\end{prop}
\begin{proof}
By its definition, $F(T)$ contains zero. Fix two elements $x,y$ in $F(T)$ and zero's open neighborhood $U \in T$. We can take zero's open neighborhood $V$ satisfying $V+V \subset U$ by the continuity of the addition at $(0,0)$.
Then, $x,y$ belong to $V$, by the definition of $F(T)$, which implies $x+y$ is in $U$. Since $U$ is an arbitrary zero's open neighborhood in $(X,T)$, we have $x+y \in F(T)$. Hence $F(T)$ is closed under the addition. A similar argument holds for closedness under the scalar multiplication, using the continuity of the scalar multiplication at $(\alpha,0)$ for all $\alpha \in K$. Therefore $F(T)$ is a subspace of $X$.\par
Next, we show that $F(T)$ is equal to ${\rm cl}_T[\{0\}]$. Take an arbitrary element $x$ from ${\rm cl}_T[\{0\}]$ and zero's open neighborhood $U \in T$. Because $x+U$ is an open neighborhood of $x$, we have $x+U \in T$ contains $0$, which implies $-x$ is in $U$. Since $U$ is an arbitrary zero's open neighborhood, $-x$ is in $F(T)$. Then, $x$ is in the subspace $F(T)$. Thus, we have ${\rm cl}_T[\{0\}] \subset G(T)$. To show the other inclusion, take an element $x$ from $F(T)$.  The subspace $F(T)$ has $-x$. Thus, $0 = x + (-x)$ is in $x+U$ for arbitrary zero's open neighborhood $U$ in $T$, which implies $x$ belongs to ${\rm cl}_T[\{0\}]$ because $\{\,x +U \mid 0 \in U \in T\,\}$ is a base of a neighborhood system at $x$. 
\end{proof}

We see some properties of $F$ and $G_0$ in the next lemma.

\begin{lem}\label{prop of F and G_0}
Let $T \in \tau(X)$ be a compatible topology with $X$ and $S$ be a subspace of $X$. Then, we have:
\begin{enumerate}
\item Every open set $U \in T$ is $F(T)$-invariant, that is, $U + F(T) =U$ holds for all $U \in T$.\nonumber 
\item $G_0(S,T) = \{\, U+S \mid U \in T\,\}$.
\end{enumerate}
\end{lem}
\begin{proof}
First, we prove 1. $U + F(T) \supset U$ holds since $F(T)$ contains $0$ from Proposition \ref{g is subspace}.  Take an element $x$ from $U +F(T)$. Then, $x$ is represented as $x = u +y$, where $u \in U$ and $y \in F(T)$. By the definition of $F(T)$, zero's open neighborhood $U-u$ contains $F(T)$. Hence $y \in U -u$ and $x = u+y \in U$. Therefore $U$ is $F(T)$-invariant.\par
Next, we prove 2. It follows from the definition of $G_0(S,T)$ and from $\ker(\pi_S)=S$, that each element of $G_0(S,T)$ is represented as $\pi_S^{-1}(V) = \pi_S^{-1}(V) +S$, where $V \in {\pi_S}_*(T)$. Then, $\pi_S^{-1}(V)$ is in $T$ by the definition of ${\pi_S}_*(T)$. Thus, $G_0(S,T)$ is contained in $\{\, U+S \mid U \in T\,\}$. Next, take an open subset $U$ from $T$. An equality $\pi_S^{-1}(\pi_S(U)) = U+S \in T$ implies that $\pi_S(U)$ is in ${\pi_S}_*(T)$ and thus, $U+S$ is in $\pi_S^*\circ {\pi_S}_*(T) = G_0(S,T)$.      
\end{proof}

The following lemma, which plays an important role in this paper states that all compatible topologies with $X$ are constructed from Hausdorff ones and that, when a compatible Hausdorff topology is unique, compatible topologies $T$ are determined by their corresponding subspaces ${\rm cl}_T[\{0\}]$ as Y. Chen described in \cite{Che}. Although essentially the same statement as this lemma is proven in (\S 5, \cite{Kel}), when the coefficient field of $X$ is $\mathbb{R}$ or $\mathbb{C}$, here we give a proof for a more general topological field.

\begin{lem}\label{key lemma}
Let $X$ be a \rm{(}possibly infinite-dimensional\rm{)} vector space over a Hausdorff topological field $K$. Then,
\begin{enumerate}
\item $G_0: \sigma(X) \times \tau^H(X) \rightarrow \tau(X)$ is surjective. \\
\item If $\tau^H(X)$ has at most one element \rm{(}equivalently, $\tau^H(X) =\{T^{\max}(X)\}$\rm{)}, the map $G :\sigma(X) \rightarrow \tau(X)$ defined by
$$
G(S) \coloneqq G_0(S,T^{\max}(X))
$$
is bijective.
Moreover, the inverse map of $G$ is $F:\tau(X) \rightarrow \sigma(X)$ and the lattice $(\tau(X), \subset)$ is  isomorphic to the lattice $(\sigma(X), \supset)$ by $F$. 
\end{enumerate}
\end{lem}

\begin{proof}
First, we show that $G_0$ is surjective. We fix $T$ as an arbitrary element of $\tau(X)$. By extending a basis of $F(T)$, the space $X$ is decomposed into $X' \oplus F(T)$, where $X' \subset X$ is a linear subspace.  We endow $X'$ with a relative topology $T_{X'}$ of $(X,T)$ and $F(T)$ with the topology $T^{\max}(F(T))$. 
We define a natural linear isomorphism $L$ by
$$
L:X' \times F(T) \ni (x',y) \mapsto x'+y \in X' \oplus F(T) = X.
$$
We obtain a topology $T_0$ on $X$ coinduced by $L$, where $X' \times F(T)$ has a product topology. Because the natural inclusion $X' \xhookrightarrow{} X$ is linear, the topology  $T_{X'}$ is a compatible topology. 
Because, for every non-zero $x'\in X'$, by the definition of $F(T)$, we can take zero's open neighborhood $U \in T$ to which $x'$ does not belong, $T_{X'}$ is Hausdorff. Thus, $(X',T_{X'})$ is a Hausdorff topological vector space. The space $(F(T),T^{\max}(F(T)))$ is also a Hausdorff space by Proposition \ref{Hausdorff}. Thus, $T_0$ is a Hausdorff topology, that is, $T_0 \in \tau^H(X)$. 
We complete a proof of 1 by showing the following equalities:
\begin{align}
G_0(F(T),T_0) =&\{\,U + F(T) \mid U \in T_0 \,\}\\
                 =&\{\,(X' \cap V) + F(T) \mid V \in T \,\}\\
                 =&\,T.
\end{align}
(1) follows from 2 of Lemma \ref{prop of F and G_0}. For (2), each element $U$ in $T_0$ is represented as 
$$
U=\bigcup_{i \in I} ((X'\cap V'_i) + V_i),
$$
where $I$ is an index set, $\emptyset \not=V'_i \in T$ and $\emptyset \not=V_i \in T^{\max}(F(T))$. Then, we have 
$$
U + F(T) = X' \cap (\bigcup_{i \in I}V'_i) + F(T).
$$
This implies that the inclusion $\subset$ in (2) holds. The opposite inclusion in (2) can be proved by taking $U \in T_0$ as $U =L((X' \cap V) \times F(T))$ for each $V \in T$. For the equality in (3), it is enough  to show that $(X' \cap V) +F(T)$ is equal to $V$ for every open set $V$ in $T$. Fix an open subset $V$ from $T$. By 1 of Lemma \ref{prop of F and G_0}, the $F(T)$-invariance of $V$ implies $(X' \cap V) +F(T) \subset V$. For the opposite inclusion, take an element $x$ from $V$. Since $X$ is decomposed into $X' \oplus F(T)$, the element $x$ is represented as $x=x'+y$, where $x' \in X'$ and $y \in F(T)$. Again, from the $F(T)$-invariance of $V$, we have $x' = x -y$ is in $V$. Thus, $x=x' + y \in (X'\cap V) +F(T)$.\par
Next, we assume that $\tau^H(X) =\{T^{\max}(X)\}$ and show that $G$ is bijective. Since $T_0$ in the proof of the first claim coincides with $T^{\max}(X)$, we have $G \circ F(T) = G_0(F(T), T^{\max}(X)) = G_0(F(T),T_0) = T$. From 2 of Lemma \ref{prop of F and G_0}, we have
$$
F \circ G (S) = \bigcap \{\,U+S \mid 0 \in U +S, U \in T^{\max}(X)\,\}= \bigcap \{\,V+S \mid 0 \in V \in T^{\max}(X)\,\}.
$$
Thus, it suffices to show that $\bigcap\{\,V+S \mid 0 \in V \in T^{\max}(X)\,\}$ is equal to $S$ to prove $F \circ G = {\rm id}_{\sigma(X)}$. It is obvious that $S$ is contained in $\bigcap \{V+S \mid 0 \in V \in T^{\max}(X)\,\}$. For an element $x$ from $X \setminus S$, because the quotient space $(X/S,T^{\max}(X/S))$ is a Hausdorff space from Proposition \ref{Hausdorff}, we can take disjoint open neighborhoods $V_1$ and $V_2$ of $\pi_S(x)$ and $\pi_S(0)$, respectively. 
Since $\pi_S^{-1}(V_2)= \pi_S^{-1}(V_2) +S$ is in $T^{\max}(X)$, which does not have an intersection with $\pi_S^{-1}(V_1)$, hence we have
$$
x\not\in\bigcap \{\,V+S \mid 0 \in V \in T^{\max}(X)\,\}.
$$   
Therefore, we have $G \circ F = {\rm id}_{\tau(X)}$ and $F \circ G = {\rm id}_{\sigma(X)}$. \par
Last, $(\tau_K(X),\subset)$ is a lattice by the argument right before Definition \ref{max topology}. The definition of $F$ and 2 of Lemma \ref{prop of F and G_0} imply that $F$ and $G$ preserve the orders of $(\tau(X),\subset)$ and $(\sigma(X),\supset)$. It is obvious that a bijection between lattices which preserve their orders is a lattice isomorphism. Therefore, a lattice $(\tau(X), \subset)$ is isomorphic to $(\sigma(X),\supset)$.   
\end{proof}

\begin{dfn}
When $\tau_K^H(X)$ has at most one element, we call the map $G:\sigma_K(X) \rightarrow \tau_K(X)$ in the above Lemma \ref{key lemma} the \it{strip map} between $\sigma_K(X)$ and $\tau_K(X)$. 
\end{dfn}

By lemma \ref{key lemma}, we can understand the lattice $(\tau_K(X),\subset)$ by a lattice of all subspaces $(\sigma_K(X), \supset)$ if $\tau_K^H(X)$ is a singleton.
However, when $\tau_K^H(X)$ has more than one element, we can not understand $(\tau_K(X), \subset)$ just by using Lemma \ref{key lemma}, since it does not state the structure of $\tau_K^H(X)$. 
Next, we give two examples of vector spaces: one example is a vector space that has a unique compatible Hausdorff topology 
and the other is a vector space which has more than one compatible Hausdorff topology

\begin{exa}
Let $K$ be a finite field which has $q$ elements with discrete topology, and $X$ be a finite-dimensional vector space over $K$. Because a Hausdorff topology on a finite point set is the discrete topology, we can use Lemma \ref{key lemma} to conclude that the number of elements of $\tau_K(X)$ is equal to that of $\sigma_K(X)$, namely
\begin{align}
\sum_{d=0}^{n} \frac{\prod_{k=1}^n (q^k-1)}{\prod_{k=1}^d (q^k-1) \prod_{k=1}^{n-d} (q^k-1)},\nonumber
\end{align}
where $n$ is a dimension of $X$ and $\prod_{k=1}^0(q^k-1)=1$. See \cite[Theorem 1]{Cha} for a proof of the cardinality of $\sigma_K(X)$ being represented as above.  
\end{exa}

We introduce a notion of a valuation field to give another example.
\begin{dfn}\label{valuation field}
Let $K$ be a commutative field. A function $\nu :K \rightarrow \mathbb{R}$ is called a {\it valuation on $K$} if $\nu$ satisfies the following four properties:
\begin{enumerate}
\item for all $\alpha \in K$, $\nu(\alpha) \geq 0$,
\item for all $\alpha \in K$, $\nu(\alpha) = 0 \Leftrightarrow \alpha = 0$,
\item for all $\alpha, \beta \in K$, $\nu(\alpha \times \beta) = \nu(\alpha)\nu(\beta)$,
\item for all $\alpha, \beta \in K$, $\nu(\alpha + \beta) \leq \nu(\alpha) + \nu(\beta)$.
\end{enumerate}
We say a valuation field $(K,\nu)$ is {\it complete} and {\it non-trivial} if a metric space $(K,d_{\nu})$ defined by $d_{\nu}(\alpha,\beta) \coloneqq\nu(\alpha - \beta)$ is complete and is not a discrete topological space, respectively.
\end{dfn}

Next, we consider a finite-dimensional vector space $X$ over a non-trivial complete valuation field $(K,\nu)$. We see that $\tau_K^H(X)$ is a singleton by the following proposition stated in \cite{Bou}.

\begin{prop}[{\rm \cite[\S 2, No.3, Theorem 2]{Bou}}]\label{Bourbaki} 
Let $(K,\nu)$ be a non-trivial complete valuation field and $X$ be a finite-dimensional Hausdorff topological vector space over $K$. For a linear basis $\{\, b_1,b_2, \dots, b_n\,\}$ of $X$, a linear map $\displaystyle K^n \ni (\alpha_1,\alpha_2,\dots,\alpha_n) \mapsto \sum_{i=1}^n \alpha_i \cdot b_i \in X$ is a homeomorphism, where $K^n$ is endowed with a product topology of $K$.
\end{prop}

As a consequence, for a finite-dimensional vector spaces over $\mathbb{R},\mathbb{C}$ and $\mathbb{Q}_p$, all compatible topologies correspond to its all subspaces.
Next, we give an example of a vector space which has more than one compatible Hausdorff topology.

\begin{exa}
We endow a real field $\mathbb{R}$ with the ordinary topology defined by the ordinary absolute value and give a relative topology to $K\coloneqq\mathbb{Q}$.
We construct two compatible Hausdorff topologies on $X\coloneqq\mathbb{Q} \times \mathbb{Q}$. One is a product topology $T_P$ and the other is constructed as follows.
Let $F\coloneqq\mathbb{Q}(\sqrt{2})$ be the smallest subfield of $\mathbb{R}$ which contains $\mathbb{Q}$ and $\sqrt{2}$.  We identify $X$ and $F$ as a $\mathbb{Q}$-vector space by a bijection map $f: X \rightarrow F$ by which $(p,q)$ is sent to $p +q\sqrt{2}$. 
We endow $F$ with a relative topology of $\mathbb{R}$ and endow $X$ with an induced topology by $f$ denoted by $T_f$. 
Then, $T_P$ and $T_f$ are both Hausdorff compatible topologies.  We show that $T_P$ does not coincide with $T_f$ as a family of open sets in $X$ by contradiction.
Assume that $T_P$ coincides with $T_f$.  
Since zero has an open neighborhood which is contained in a set $B\coloneqq\{\,(p,q) \in X \mid p^2 +q^2 <1\,\}$ in $(X,T_P)$, there is a small positive $\epsilon$ such that $f^{-1}(B'(0,\epsilon))$ is contained in $B$, where $B'(0, \epsilon)$ is an open ball of radius $\epsilon$ and centered at $0$ in $\mathbb{R}$. Fix a sufficiently large positive integer $m$ so that $10^{-m} < \epsilon$ holds. We define a point $(p_0,q_0)$ in $X$ by
$$
p_0 \coloneqq -\frac{1}{10^{m}}\lfloor 10^m\sqrt{2} \rfloor,~q_0 \coloneqq 1,
$$
where $\lfloor \cdot \rfloor$ is the floor function by which a real number $x$ is sent to an integer that is less than or equal to $x$.  Then, $f(p_0,q_0)$ is in $B'(0,\epsilon)$, and is in $B$. This is a contradiction because we obviously have $1 \leq {p_0}^2+{q_0}^2$, which implies that the point $(p_0,q_0)$ does not belong to $B$.   
\end{exa}

%%%%%%%%%%%%%%%%%%%%%%%%%%%%%%%%%%%%%%%%%%%%%%%%%%%%%%%%%%%%%%%%%%%%%%%%%%%%%%%%%%%%%%%%%%%%%%%%%%%%%%%%%%%%%%%%%%%%%%%%%%%%%%%%%%%%%%%%%%%%%%%%%%%%%%%%%%%%%%%%%%%%%%%%%%%%%%%%%%%%%%section 3
\section{Topologies on a Vector Space over a Non-complete Valuation Field }
\begin{center}
{\bf For the rest of this paper, we only consider $K$ to be a valuation field $(K,\nu)$.}
\end{center}
Our main theorem is an analogy of 2 of Lemma \ref{key lemma} without the assumption that $\tau_K^H(X)$ is a singleton. 
First, we conduct a metric completion of a valuation field.
Let $(K,\nu)$ be a non-trivial valuation field (see Definition \ref{valuation field}). We denote by $\hat{K}$ the completion of $K$ as a metric space $(K,d_{\nu})$. 
Then, $\hat{K}$ has natural addition, $+$ and multiple, $\times$ defined by  
\begin{align}
\hat{\alpha} + \hat{\beta} \coloneqq& \lim_{n \to \infty} \alpha_n + \beta_n, \nonumber \\
\hat{\alpha} \times \hat{\beta} \coloneqq& \lim_{n \to \infty} \alpha_n \times \beta_n, \nonumber
\end{align}
for all $\hat{\alpha},\hat{\beta} \in \hat{K}$, where $\{\alpha_n\}_{n=1}^{\infty}$ and $\{\beta_n\}_{n=1}^{\infty}$ are sequences in the dense subset $K \subset \hat{K}$ which converge to $\hat{\alpha}$ and $\hat{\beta}$, respectively. 
$\hat{K}$ is a field extension of $K$, and an extension $\hat{\nu}$ of the valuation $\nu$ to $\hat{K}$ is defined by 
$$
\displaystyle \hat{\nu}(\hat{\alpha})\coloneqq \lim_{n \to \infty} \nu(\alpha_n)
$$
for all $\hat{\alpha} \in \hat{K}$, where $\{\alpha_n\}_{n=1}^{\infty}$ is a sequence in $K$ which converges to $\hat{\alpha} \in \hat{K}$. Now we have a complete valuation field $(\hat{K}, \hat{\nu})$.
For a given {\bf finite-dimensional} vector space $X$ over $K$, we denote, by $\hat{X}$, a tensor product $\hat{K} \bigotimes_K X$, which is a $\hat{K}$-vector space and denote, by $I: X \rightarrow \hat{X}$, an injective map defined by $x \mapsto 1 \otimes x$.\\

\begin{dfn} \label{FS}
We define maps $\hat{F}:\tau_K(X) \rightarrow \sigma_{\hat{K}}(\hat{X})$ and $\hat{G}:\sigma_{\hat{K}}(\hat{X}) \rightarrow \tau_K(X)$ by
\begin{align}
\hat{F}(T) \coloneqq&\bigcap_{0 \in U \in T} {\rm cl}_{T_{\hat{K}}^{\max}(\hat{X})}[I(U)],\nonumber \\
\hat{G}(S) \coloneqq&I^*(G(S)), \nonumber 
\end{align}
where $G$ is the strip map between $\sigma_{\hat{K}}(\hat{X})$ and $\tau_{\hat{K}}(\hat{X})$.
\end{dfn}

The topology $\hat{G}(S)$ is an element of $\tau_K(X)$ by the following reasons:
$\hat{X}$ is a $K$-vector space by restricting the scalar multiplication to $K \times \hat{X}$.
Therefore, $G(S)$ is an element of $\tau_K(\hat{X})$.
Since $\hat{G}(S)$ is a coinduced topology by the $K$-linear map $I: X \rightarrow \hat{X}$, we have $\hat{G}(S) \in \tau_K(X)$.\\

We show that $\hat{F}(T)$ is actually a $\hat{K}$-subspace of $\hat{X}$.

\begin{lem}\label{hat S is a subspace}
For each $T \in \tau_K(X)$, a subset $\hat{F}(T)$ of $\hat{X}$ is a $\hat{K}$-subspace of $\hat{X}$.
\end{lem}
\begin{proof}
Obviously, the zero of  $\hat{X}$ belongs to $\hat{F}(T)$.\par 
We show that $\hat{F}(T)$ is closed under the addition and the scalar multiplication of $\hat{X}$.
Take arbitrary elements $x,y$ from $\hat{F}(T)$ and fix an open neighborhood $U$ of zero of $X$. 
It suffices to show that for any zero's open neighborhood $V$ in $T_{\hat{K}}^{\max}(\hat{X})$, the intersection of $x+y+V$ and $I(U)$ is not empty, since a family $\{\,x+y+V \mid 0 \in V \in T_{\hat{K}}^{\max}(\hat{X})\,\}$ is a base of the neighborhood system at $x+y$ in $T_{\hat{K}}^{\max}(\hat{X})$.
By the continuity of the additions of $X$ and $\hat{X}$ with respect to $T$ and $T_{\hat{K}}^{\max}(\hat{X})$, respectively, there are open neighborhoods of zeros $T \ni U_1,U_2 \subset X$ and $T_{\hat{K}}^{\max}(\hat{X}) \ni V_1,V_2 \subset \hat{X}$ satisfying
$U_1 +U_2 \subset U$ and $V_1 +V_2 \subset V$. Because $x +V_1$ and $y +V_2$ are neighborhoods of $x$ and $y$, respectively, they intersect $I(U_1)$ and $I(U_2)$.
Hence $(x+y+V) \cap I(U)$ is not empty, containing $((x+V_1)\cap I(U_1)) + ((y+V_2)\cap I(U_2))$.\par
For the scalar multiplication, fix an element $\alpha \in \hat{K}$ and $x \in \hat{F}(T)$. If $\alpha$ is $0$, then
 $\alpha \cdot x$ is $0$ and belongs to $\hat{F}(T)$. If $\alpha$ is not $0$, for an arbitrary zero's open neighborhood $U \in T$ and neighborhood $V \in T_{\hat{K}}^{\max}(\hat{X})$ of $\alpha \cdot x$, there are open sets $L \subset \hat{K}$ and $V' \in T_{\hat{K}}^{\max}(\hat{X}) $ such that 
$\alpha \in L,~x \in V'$ and $L \cdot V' \subset V$ because of the continuity of the scaler multiplication. Since $K$ is a dense subset of the Hausdorff space $\hat{K}$, we have an element $q \in L \cap (K\setminus \{0\})$.
Now $I(q^{-1}\cdot U)$ intersects $V'$ because $q^{-1}\cdot U$ is an open neighborhood of $0 \in X$. Let $I(q^{-1} \cdot u)$ be an element of the intersection. $I(u) = q \cdot I(q^{-1}\cdot u) \in L \cdot V'$ belongs to $I(U) \cap V$.
Thus, $\alpha \cdot x$ belongs to ${\rm cl}_{T^{\max}(\hat{X})}[I(U)]$. Since $U$ is arbitrary, we have $\alpha \cdot x$ belongs to $\hat{F}(T)$ and $\hat{F}(T)$ is a $\hat{K}$-subspace.  
\end{proof} 

With the above preparation, our main theorem is stated as following way.

\begin{thm}\label{main thm}
We have the following:
\begin{enumerate}
\item $\hat{F}$ and $\hat{G}$ invert the inclusions in $\sigma_{\hat{K}}(\hat{X})$ and $\tau_K(X)$.
\item $\hat{F} \circ \hat{G} = {\rm id}_{\sigma_{\hat{K}}(\hat{X})}$.
\item $T \subset \hat{G} \circ \hat{F} (T)$ for all $T \in \tau_K(X)$.  
\item If $\hat{K}$ is a locally compact space with respect to $d_{\hat{\nu}}$, 
 then $\hat{G} \circ \hat{F} =  {\rm id}_{\tau_K(X)}$. Thus,  the lattice $(\tau_K(X),\subset)$ is lattice isomorphic to $(\sigma_{\hat{K}}(\hat{X}),\supset)$ by $\hat{F}$.
\end{enumerate}  
\end{thm}

\begin{rmk}
In the forth statement of the theorem, it is known that all non-trivial locally compact topological fields are finite extension of either the field of real numbers $\mathbb{R}$, the fields of $p$-adic numbers $\mathbb{Q}_p$ or the fields of formal Laurent series $\mathbb{F}_p((t))$ over the finite fields $\mathbb{F}_p$, where $p$ ranges over prime numbers (See \cite[Theorem 22]{Pon} for a proof). Thus, the field $K$ is a subfield of one of them.  
\end{rmk}
%%%%%%%%%%%%%%%%%%%%%%%%%%%%%%%%%%%%%%%%%%%%%%%%%%%%%%%%%%%%%%%%%%%%%%%%%%%%%%%%%%%%%%%%%%%%%%%%%%%%%%%%%%%%%%%%%%%%%%%%%%%%%%%%%%%%%%%%%%%%%%%%%%%%%%%%%%%%%%%%%%%%%%%%%%%%%%%%%%%%%%section 4
\section{Proof of the Main Theorem}

We need several lemmas.

\begin{lem} \label{X is dense in hat X}
$I(X)$ is a dense subset of $\hat{X}$ with respect to arbitrary compatible topology $T \in \tau_{\hat{K}}(\hat{X})$.
\end{lem}
\begin{proof}
Let $\hat{x}$ be an element of $\hat{X}$. 
We represent $\hat{x}$ as 
$\sum_{k=1}^n \alpha_k \otimes x_k~~(\alpha_k \in \hat{K},x_k \in X)$. 
By applying Lemma \ref{hat S is a subspace} for $X$ endowed with the indiscrete topology, $T_i$, we obtain that $F(T_i) ={\rm cl}_{T^{\max}(\hat{X})}[I(X)]$ is a $\hat{K}$-subspace of $\hat{X}$.
Thus, ${\rm cl}_{T^{\max}(\hat{X})}[I(X)]$ is closed under the scalar multiplication and the addition, which  implies that $\hat{x}$ belongs to ${\rm cl}_{T^{\max}(\hat{X})}[I(X)]$. Thus, $I(X)$ is a dense subset in $\hat{X}$ with respect to $T_{\hat{K}}^{\max}(\hat{X})$ and it is also dense with respect to arbitrary weaker topology $T$ than $T_{\hat{K}}^{\max}(\hat{X})$.
\end{proof}

\begin{lem}\label{good neighborhood}
For each $T \in \tau_K(X)$, there is an open neighborhood $U \in T$ of zero in $X$, which satisfies the following property:
If a subspace $S$ of $\hat{X}$ is contained in ${\rm cl}_{T^{\max}}[I(U)]$, then $S \subset \hat{F}(T)$ holds.
\end{lem}
\begin{proof}
For a zero's open neighborhood $U \in T$, we define a non-negative integer ${\rm Md}(U)$ by
$$
{\rm Md}(U) \coloneqq \max\{\,\dim_{\hat{K}}(S) \mid \text{$S$ is a $\hat{K}$-subspace contained in ${\rm cl}_{T^{\max}}[I(U)]$}\,\}.
$$
Let $m$ be the minimum number of ${\rm Md}(U)$'s for all zero's open neighborhoods $U \in T$. We take a zero's open neighborhoods $U_0,U_1 \in T$ so that ${\rm Md}(U_0)$ attains $m$ and that $U_1+U_1  \subset U_0$ by the continuity of the addition at $(0,0)$.
Because $U_1 \subset U_0$, we have $m \leq {\rm Md}(U_1) \leq {\rm Md}(U_0)=m$ and hence we can take a subspace $S_1$ contained in ${\rm cl}_{T^{\max}}[I(U_1)]$ whose dimension is $m$. To prove the lemma, it suffices to show that every subspace $S$ contained in ${\rm cl}_{T^{\max}}[I(U_1)]$ is also contained in $S_1$ and that $S_1$ is equal to $\hat{F}(T)$. For the first claim, take a subspace $S$ which is contained in ${\rm cl}_{T^{\max}}[I(U_1)]$. Since $S_1$ and $S$ are contained in ${\rm cl}_{T^{\max}}[I(U_1)]$, a subspace $S+S_1$ is  contained in ${\rm cl}_{T^{\max}}[I(U_0)]$. Here we use the inclusion ${\rm cl}_{T^{\max}}[A] + {\rm cl}_{T^{\max}}[B] \subset{\rm cl}_{T^{\max}}[A+B]$, which is proved right after Definition \ref{Minkowski sum}. By the definition of ${\rm Md}(U_0)$, the subspace $S+S_1$ is $m$-dimensional, which contains the $m$-dimensional subspace $S_1$. Thus, $S+S_1$ is equal to $S_1$, which implies $S \subset S_1$. For the second claim, since $\hat{F}(T)$ is contained in ${\rm cl}_{T^{\max}}[I(U_1)]$, a subspace $\hat{F}(T)$ is contained in $S_1$ from the first claim. We prove the other inclusion by contradiction. Assume that we have an element $x$ from $S_1\setminus \hat{F}(T)$. By the definition of $\hat{F}(T)$, there is a zero's open neighborhood $U_2$ such that $x$ does not belong to ${\rm cl}_{T^{\max}}[I(U_2)]$. Take a subspace $S_2$ contained in ${\rm cl}_{T^{\max}}[I(U_1 \cap U_2)]$ whose dimension attains ${\rm Md}(U_1 \cap U_2)$. Since $S_2$ and $S_1$ is contained in ${\rm cl}_{T^{\max}}[I(U_1)]$, a subspace $S_1+S_2$ is contained in ${\rm cl}_{T^{\max}}[I(U_0)]$. Here we again use the inclusion ${\rm cl}_{T^{\max}}[A] +{\rm cl}_{T^{\max}}[B] \subset {\rm cl}_{T^{\max}}[A+B]$.  
Thus, we have the following inequality:
$$
m \leq {\rm Md}(U_1 \cap U_2) = \dim(S_2) \leq \dim (S_1+S_2) \leq {\rm Md}(U_0) = m. 
$$
This implies $S_1+S_2$ is equal to $S_2$, which is a contradiction because  $x \in S_1+S_2 = S_2 \subset {\rm cl}_{T^{\max}}[I(U_2)]$.     
\end{proof}

\begin{dfn}\label{good norms}
Let $\{\,b_1,b_2, \dots ,b_n\,\}$ be a $K$-linear basis of $X$.  Then, $\{\,I(b_1),I(b_2), \dots ,I(b_n)\,\}$ is a $\hat{K}$-linear basis of $\hat{X}$.
We define norms $\|\cdot\|_X$ and $\|\cdot\|_{\hat{X}}$ on $X$ and $\hat{X}$, respectively, defined by
\begin{align}
\left\|\sum_{k=1}^n \alpha_k \cdot b_k\right\|_X \coloneqq& \sum_{k=1}^n \nu(\alpha_k), \nonumber \\
\left\|\sum_{k=1}^n \hat{\alpha}_k \cdot I(b_k)\right\|_{\hat{X}} \coloneqq& \sum_{k=1}^n \hat{\nu}(\hat{\alpha}_k). \nonumber
\end{align}
 
\end{dfn}
Because $I: X \rightarrow \hat{X}$ preserves coefficients, $I$ preserves these norms.
Moreover, the topologies which these norms define coincide with $T_K^{\max}(X)$ and $T_{\hat{K}}^{\max}(\hat{X})$, respectively.
This is proved as follows:
Let $T_X$ be the topology which $\|\cdot\|_X$ defines.
It is easily checked that $T_X \in \tau_K(X)$. Thus, $T_X \subset T_K^{\max}(X)$ follows from the definition of $T_K^{\max}(X)$.
Next, let $V$ be a zero's neighborhood in $T_K^{\max}(X)$. By the continuity of the addition, there is a zero's open neighborhood $U \in T_K^{\max}(X)$ satisfying
$$
\underbrace{U + U + \cdots +U}_{n-times} \subset V.
$$
For each $k=1,2,\dots,n$, by the continuity of the scalar multiplication at points $(0,b_k) \in K \times X$, we have $\epsilon_k >0$ so that $\nu(\alpha) < \epsilon_k$ implies $\alpha \cdot b_k \in U$. We define $\epsilon \coloneqq \min\{\,\epsilon_k \mid k=1,2,\dots, n\,\}$. Then, an open ball $B\coloneqq\{\,x \in X \mid \|x\|_X < \epsilon \,\}$ is a zero's open neighborhood with respect to $T_X$ contained in $V$. Therefore, we have $T_K^{\max}(X) \subset T_X$. A similar argument shows that $\|\cdot\|_{\hat{X}}$ defines a topology which coincides with $T_{\hat{K}}^{\max}(\hat{X})$.

\begin{lem}\label{regular open}
Let $A$ be a subset of $X$. Then, the following equalities hold:
\begin{align}
I({\rm cl}_{T_K^{\max}(X)}[A]) =& I(X) \cap {\rm cl}_{T_{\hat{K}}^{\max}(\hat{X})}[I(A)],\nonumber\\
I({\rm int}_{T_K^{\max}(X)}[{\rm cl}_{T_K^{\max}(X)}[A]])=& I(X) \cap {\rm int}_{T_{\hat{K}}^{\max}(\hat{X})}[{\rm cl}_{T_{\hat{K}}^{\max}(\hat{X})}[I(A)]].\nonumber
\end{align}
\end{lem}

\begin{proof}
Because $I:(X,T_K^{\max}(X)) \rightarrow (\hat{X},T_{\hat{K}}^{\max}(\hat{X}))$ is a continuous $K$-linear map, an inclusion 
$$
I({\rm cl}_{T^{\max}(X)}[A]) \subset I(X) \cap {\rm cl}_{T^{\max}(\hat{X})}[I(A)]
$$
 holds. 
To prove the other inclusion, we take norms defined in Definition \ref{good norms} denoted by $\|\cdot\|_X$ and $\|\cdot\|_{\hat{X}}$.
Take an element from $I(X) \cap {\rm cl}_{T^{\max}(\hat{X})}[I(A)]$ denoted by $I(x)$. Since $I(x)$ belongs to the closure of $I(A)$, there is a sequence $\{a_n\}_{n=1}^{\infty}$ in $A$ with $\| I(x) -I(a_n) \|_{\hat{X}} \rightarrow 0~(n \rightarrow \infty)$.
Since $I$ preserves the norms, $\| x - a_n \|_X$ converges to $0$, which implies that $x$ is in the closure of $A$ and that $I(x)$ is in $I({\rm cl}_{T^{\max}(X)}[A])$. \par
For the second claim, from the injectivity of $I$ and from the first equality, we have ${\rm cl}_{T^{\max}(X)}[A] = I^{-1}({\rm cl}_{T^{\max}(\hat{X})}[I(A)])$. This equality and the continuity of $I:(X,T_K^{\max}(X)) \rightarrow (\hat{X},T_{\hat{K}}^{\max}(\hat{X}))$ imply that the set 
$$
I^{-1}({\rm int}_{T^{\max}(\hat{X})}[{\rm cl}_{T^{\max}(\hat{X})}[I(A)]])
$$
 is open that is contained in ${\rm cl}_{T^{\max}(X)}[A]$. Therefore, 
$$
{\rm int}_{T^{\max}(X)}[{\rm cl}_{T^{\max}(X)}[A]]  \supset I^{-1}({\rm int}_{T^{\max}(\hat{X})}[{\rm cl}_{T^{\max}(\hat{X})}[I(A)]])
$$
 holds and, by taking the image of $I$, we have the inclusion $\supset$ in the second claim. Next, for an element $I(x)$ from $I({\rm int}_{T^{\max}(X)}[{\rm cl}_{T^{\max}(X)}[A]])$, we have an open ball $B_X(x,\epsilon)$, defined by the norm $\|\cdot\|_X$, whose center is $x$ and radius is $\epsilon$ contained in ${\rm int}_{T^{\max}(X)}[{\rm cl}_{T^{\max}(X)}[A]]$. From the first equality and from $I$ being an isometry, $B_{\hat{X}}(I(x),\epsilon) \cap I(X) = I(B_X(x,\epsilon)) \subset {\rm cl}_{T^{\max}(\hat{X})}[I(A)]$ holds, where $B_{\hat{X}}(I(x),\epsilon)$ is an open ball whose center is $I(x)$, radius is $\epsilon$ with respect to the norm $\|\cdot\|_{\hat{X}}$. By Lemma \ref{X is dense in hat X}, $I(X)$ is a dense subset of $\hat{X}$, which implies that $B_{\hat{X}}(I(x),\epsilon) \subset {\rm cl}_{T^{\max}(\hat{X})}[B_{\hat{X}}(I(x),\epsilon) \cap I(X)] \subset {\rm cl}_{T^{\max}(\hat{X})}[I(A)]$. Therefore, we have the other inclusion $\subset$ in the second claim.     
\end{proof}

\begin{lem}\label{identity on sigma}
We have $\hat{F} \circ \hat{G} = {\rm id}_{\sigma_{\hat{K}}({\hat{X}})}$.
\end{lem}

\begin{proof}
Let $S$ be a $\hat{K}$-linear subspace of $\hat{X}$ and put $\hat{T} \coloneqq G(S)$, where $G:\sigma_{\hat{K}}(\hat{X}) \rightarrow \tau_{\hat{K}}(\hat{X})$ is the strip map. It suffices to show that $\hat{F} \circ \hat{G}(S)$ coincides with $F(\hat{T})$ because we have $F(\hat{T}) = F \circ G(S) = S$ by Lemma \ref{key lemma} and Proposition \ref{Bourbaki}. 
First, the definitions of $\hat{F}$ and $\hat{G}$ show that:
\begin{align}
\hat{F} \circ \hat{G}(S) =& \bigcap_{0 \in U \in \hat{G}(S)}{\rm cl}_{T^{\max}(\hat{X})}[I(U)] \nonumber \\
                                =& \bigcap_{0 \in V \in \hat{T}}{\rm cl}_{T^{\max}(\hat{X})}[I(X) \cap V] \nonumber.
\end{align}
Next, for all zero's open neighborhood $V$ in $\hat{T}$, because $I(X)$ is a dense subset of $\hat{X}$ and, because $V \in \hat{T} \subset T^{\max}(\hat{X})$ holds, we have $V \subset {\rm cl}_{T^{\max}(\hat{X})}[I(X) \cap V]$. Therefore, $F(\hat{T}) \subset \bigcap_{0 \in V \in \hat{T}}{\rm cl}_{T^{\max}(\hat{X})}[I(X) \cap V]$. We show the opposite inclusion. Let $x$ be an element of $\bigcap_{0 \in V \in \hat{T}}{\rm cl}_{T^{\max}(\hat{X})}[I(X) \cap V]$ and $V'$ be a zero's open neighborhood with respect to $\hat{T}$. 
Because of the continuity of the addition with respect to $\hat{T}$, there are zero's open neighborhoods $V_1,V_2 \in \hat{T}$ such that $V_1 -V_2$ is contained in $V'$. Then, $x$ belongs to ${\rm cl}_{T^{\max}(\hat{X})}[I(X) \cap V_1]$, and thus, we can take an element $x'$ from the intersection of $V_2 +x$ and $I(X) \cap V_1$. Then, $x$ is represented as $x = x' - (x'-x)$ and this implies that $x \in V_1 -V_2 \subset V'$. Since $V'$ is an arbitrary zero's open neighborhood with respect to $\hat{T}$, we have $F(\hat{T}) = \bigcap_{0 \in V \in \hat{T}}{\rm cl}_{T^{\max}(\hat{X})}[I(X) \cap V]$.  
\end{proof}

\begin{lem}\label{T is contained in GF}
We have $T \subset \hat{G} \circ \hat{F}(T)$ for all $T \in \tau_K(X)$.
\end{lem}
\begin{proof}
Since the addition of $X$ is continuous with respect to $T$ and $\hat{G} \circ \hat{F}(T)$, it is enough to compare neighborhoods of zero to show this lemma. More precisely, we show that for each zero's open neighborhood $U \in T$, there is an zero's open neighborhood $U' \in \hat{G} \circ \hat{F}(T)$ contained in $U$. By the continuity of the addition, zero has an open neighborhood $U_0 \in T$ satisfying $U_0+U_0+U_0 \subset U$.  By Lemma \ref{regular open},
 $V_0 \coloneqq{\rm int}_{T^{\max}(\hat{X})}[{\rm cl}_{T^{\max}(\hat{X})}[I(U_0)]]$ contains $I({\rm int}_{T^{\max}(X)}[{\rm cl}_{T^{\max}(X)}[U_0]])$, and hence zero belongs to $V_0$. We define $V_1$ and $U'$ by
\begin{align}
V_1 \coloneqq& V_0 + \hat{F}(T),\nonumber \\
U'   \coloneqq& I^{-1}(V_1). \nonumber
\end{align}
The definition of $U'$ and 2 of Lemma \ref{prop of F and G_0} implies that $V_1$ is an open neighborhood of zero with respect to $G(\hat{F}(T))$ and that $U'$ is an open neighborhood of zero in $\hat{G}(\hat{F}(T))$. We show that $U'$ is contained in $U$.  Since ${\rm cl}_{T^{\max}(\hat{X})}[I(U_0)]$ includes $\hat{F}(T)$, we have 
$$
V_1 \subset V_0 + {\rm cl}_{T^{\max}(\hat{X})}[I(U_0)] \subset {\rm cl}_{T^{\max}(\hat{X})}[I(U_0+U_0)],
$$
where we use the inclusion ${\rm cl}_{T^{\max}(\hat{X})}[A] + {\rm cl}_{T^{\max}(\hat{X})}[B] \subset {\rm cl}_{T^{\max}(\hat{X})}[A+B]$  that is proved right after Definition \ref{Minkowski sum}. 
By taking the intersection with $I(X)$ and by Lemma \ref{regular open}, we have the following:
$$
V_1 \cap I(X) \subset {\rm cl}_{T^{\max}(\hat{X})}[I(U_0)+I(U_0)] \cap I(X) = I({\rm cl}_{T^{\max}(X)}[U_0+U_0]) \subset I(U_0)+I(U_0)+I(U_0) \subset I(U).
$$
By taking the inverse image of $I$, we can deduce that $U' \subset U$.
\end{proof}

\begin{lem}\label{s invariant}
For $T \in \tau_K(X)$ and $U \in T$, a set ${\rm cl}_{T^{\max}(\hat{X})}[I(U)]$ is $\hat{F}(T)$-invariant, that is, the equality,
$$
{\rm cl}_{T^{\max}(\hat{X})}[I(U)] + \hat{F}(T) = {\rm cl}_{T^{\max}(\hat{X})}[I(U)]
$$
holds. 
\end{lem}

\begin{proof}
Since $\hat{F}(T)$ is a subspace of $\hat{X}$, it is enough to show that ${\rm cl}_{T^{\max}(\hat{X})}[I(U)] +s \subset {\rm cl}_{T^{\max}(\hat{X})}[I(U)]$ for all $s \in \hat{F}(T)$. Fix an element $u$ in $U$. By the continuity of the addition at $(u,0)$, there is a zero's open neighborhood $U'$ in $T$ such that $u + U' \subset U$. By  taking the closure of both sides of $I(u) + I(U') \subset I(U)$, we have $I(u) + {\rm cl}_{T^{\max}(\hat{X})}[I(U')] \subset {\rm cl}_{T^{\max}(\hat{X})}[I(U)]$ and therefore, $I(u) +s \in {\rm cl}_{T^{\max}(\hat{X})}[I(U)]$ holds. Since $u \in U$ is arbitrary, we have $I(U) +s \subset {\rm cl}_{T^{\max}(\hat{X})}[I(U)]$. By taking the closures of both side, the inclusion ${\rm cl}_{T^{\max}(\hat{X})}[I(U)] +s \subset{\rm cl}_{T^{\max}(\hat{X})}[I(U)]$ holds.
\end{proof}

\begin{dfn}
A neighborhood $N$ of zero in topological vector space over a valuation field $(K,\nu)$ is called {\it balanced neighborhood} if $N$ satisfies the following property:
\begin{align}
\text{$\nu(\kappa)\leq 1$ implies $\kappa \cdot  N \subset N$ for all $\kappa \in K$}.\nonumber
\end{align}

\end{dfn}

\begin{lem}[{\rm \cite[\S 1, No.5]{Bou}}]\label{balanced} 
A family of all balanced neighborhoods is a neighborhood base of zero in topological vector space $(X,T)$ over a non-trivial valuation field $(K,\nu)$.
\end{lem}

\begin{proof}
Fix an arbitrary open neighborhood $U$ of zero. We define the subset $N$ of $X$ by
\begin{align}
N \coloneqq \bigcap_{\nu(\kappa) \geq 1}\kappa \cdot U.\nonumber
\end{align}
It suffices to show that $N$ is a neighborhood of zero.  By the continuity of the scalar multiplication at $(0,0) \in K \times X$, there are zero's open neighborhood $V \in T$ and $\epsilon >0$ such that $B_{\nu}(0,\epsilon) \cdot V \subset U$, where $B_{\nu}(0,\epsilon)$ is an open ball with respect to $\nu$ whose center is $0$ and radius is $\epsilon$ in $K$. Since the valuation field $(K,\nu)$ is non-trivial, there is a non-zero element $\alpha \in K$ satisfying $\nu(\alpha) < \epsilon$. Then, it is easy to show that zero's open neighborhood $\alpha \cdot V$ is contained in $N$. Thus, $N$ is actually a zero's open neighborhood in $(X,T)$ contained in $U$. 
\end{proof}

\setcounter{thm}{0}
\begin{thm}[Restatement]
Let $(K,\nu)$ be a non-trivial valuation field and $X$ be a finite-dimensional vector space over $K$. 
Then, the maps $\hat{G}:\sigma_{\hat{K}}(\hat{X}) \rightarrow \tau_K(X)$ and $\hat{F}:\tau_K(X) \rightarrow \sigma_{\hat{K}}(\hat{X})$ defined in Definition \ref{FS} satisfy the following properties:
\begin{enumerate}
\item $\hat{F}$ and $\hat{G}$ invert the inclusion relation $\subset$ on $\tau_K(X)$ and $\sigma_{\hat{K}}(\hat{X})$. 
\item $\hat{F} \circ \hat{G} = {\rm id}_{\sigma_{\hat{K}}(\hat{X})}$.
\item For all $T \in \tau_{K}(X)$, we have $T \subset \hat{G} \circ \hat{F}(T)$.
\item If the completion $(\hat{K}, \hat{\nu})$ is a locally compact space, then $\hat{G} \circ \hat{F} = {\rm id}_{\tau_K(X)}$. 
Thus, the lattice $(\tau_K(X), \subset)$ is lattice isomorphic to $(\sigma_{\hat{K}}(\hat{X}), \supset)$ by $\hat{F}$.
\end{enumerate}
\end{thm}
\begin{proof}
For 1, the map $G$ inverts the inclusion from Lemma \ref{prop of F and G_0} and thus, the definition of $\hat{G}$ implies that $\hat{G}$ also inverts the inclusion. The map $\hat{F}$ inverts the inclusion by its definition. We have already proved 2 and 3 in Lemma \ref{identity on sigma} and Lemma \ref{T is contained in GF}. We prove 4 in particular, $\hat{G} \circ \hat{F}(T) \subset T$. Fix an element $T$ of $\tau_K(X)$.
By Lemma \ref{good neighborhood}, there is a zero's open neighborhood $U \in T$ such that if a subspace $S$ is contained in ${\rm cl}_{T^{\max}(\hat{X})}[U]$, then $S$ is also contained in $\hat{F}(T)$. By Lemma \ref{balanced}, we have a balanced neighborhood $N$ contained in $U$, and we take a zero's {\bf open} neighborhood contained in $N$ denoted by $U'$ with respect to $T$. Let $\hat{Y}$ be the quotient space $\hat{X}/\hat{F}(T)$ and $P:\hat{X} \rightarrow \hat{Y}$ be the quotient map. \par
We claim that, for every non-zero $\hat{y} \in \hat{Y}$, there exists a positive number $M_{\hat{y}}$ depending on $\hat{y}$, such that for any $\alpha \in \hat{K}$ satisfying $\hat{\nu}(\alpha) \geq M_{\hat{y}}$, we have $\alpha \cdot \hat{y} \not\in {\rm cl}_{T^{\max}(\hat{Y})}[P \circ I(U')]$.
We give a proof by contradiction.  
Suppose that there is a point $\hat{y} \in\hat{Y}$ such that $\alpha \cdot \hat{y}$ belongs to ${\rm cl}_{T^{\max}(\hat{Y})}[P \circ I(U')]$ for arbitrary large $\alpha$ with respect to the valuation $\hat{\nu}$. The point $\hat{y}\not=0$ is represented as $\hat{y}=P(\hat{x}),~\hat{x} \in \hat{X} \setminus \hat{F}(T)$. By Lemma \ref{s invariant}, ${\rm cl}_{T^{\max}(\hat{Y})}[P \circ I(U')]$ is equal to $P({\rm cl}_{T^{\max}(\hat{X})}[I(U')])$. More precisely, by the continuity of $P : (\hat{X},T_{\hat{K}}^{\max}(\hat{X})) \rightarrow (\hat{Y},T_{\hat{K}}^{\max}(\hat{Y}))$, we have the inclusion $\supset$. Because an open set $P(\hat{X} \setminus {\rm cl}_{T^{\max}(\hat{X})}[I(U')])$ does not intersect $P({\rm cl}_{T^{\max}(\hat{X})}[I(U')])$ from Lemma \ref{s invariant}, $P({\rm cl}_{T^{\max}(\hat{X})}[I(U')])$ is a closed subset containing $P\circ I(U')$. Thus, we have ${\rm cl}_{T^{\max}(\hat{Y})}[P \circ I(U')] = P({\rm cl}_{T^{\max}(\hat{X})}[I(U')])$. This implies that $\alpha \cdot \hat{y} \in {\rm cl}_{T^{\max}(\hat{Y})}[P \circ I(U')]$ is equivalent to $\alpha \cdot \hat{x} \in {\rm cl}_{T^{\max}(\hat{X})}[I(U')]$. By the assumption on $\hat{y}$, we have $\alpha \cdot \hat{x} \in {\rm cl}_{T^{\max}(\hat{X})}[I(U')]$ for arbitrary large $\alpha$, and hence a subspace generated by $\hat{x}$ is contained in ${\rm cl}_{T^{\max}(\hat{X})}[I(U)]$ because $N$ is a balanced neighborhood. Because of the way we take $U$, the point $\hat{x}$ is in $\hat{F}(T)$ and this contradicts against $\hat{y}$ being non-zero.\par
Now, we again take a balanced neighborhood $N'$ contained in $U'$. By extending a basis of $\hat{F}(T)$, we obtain a subspace $S'$ to decompose $\hat{X}$ into
$
\hat{X} = \hat{F}(T) \oplus S'.
$ 
Let $\{\,b_1,b_2, \dots ,b_n\,\}$ be a basis of $S'$. Then, $\{\,P(b_1),P(b_2), \dots ,P(b_n)\,\}$ is a basis of $\hat{Y}$ and we introduce a norm $\|\cdot\|_{\hat{Y}}$ in $\hat{Y}$ as the same way in Lemma \ref{good norms} with respect to this basis. 
Since $\nu$ is non-trivial, we fix an element $\kappa$ with $\nu(\kappa) > 1$.
We show $P({\rm cl}_{T^{\max}(\hat{X})}[I(N')])$ is bounded, that is, there exists a positive $M>0$ such that $B_{\hat{Y}}(0,\nu(\kappa^{M+1}))$ contains $P({\rm cl}_{T^{\max}(\hat{X})}[I(N')])$, where $B_{\hat{Y}}(q,r)$ is an open ball centered at $q$ whose radius is $r$ with respect to the norm $\|\cdot\|_{\hat{Y}}$. We define a subset $A$ of $\hat{Y}$ as follows:
$$
A \coloneqq \{\,\hat{y} \in \hat{Y} \mid 1 \leq \|\hat{y}\|_{\hat{Y}} \leq \nu(\kappa)\,\}.
$$   
{\bf We use the assumption of $\hat{K}$ being locally compact to deduce that $A$ is a compact subset in $(\hat{Y},T_{\hat{K}}^{\max}(\hat{Y}))$.}
For every element $\hat{y}$ in $A$, we can take a natural number $n_{\hat{y}}$ so that $\kappa^{n_{\hat{y}}} \geq M_{\hat{y}}$ holds and positive number $\epsilon_{\hat{y}}$ so that $B_{\hat{Y}}(\kappa^{n_{\hat{y}}} \cdot \hat{y}, \epsilon_{\hat{y}}) \cap {\rm cl}_{T^{\max}(\hat{Y})}[P \circ I(N')] = \emptyset$ holds. Since $A$ is compact, we can take a finite subcover of $\{\,\kappa^{-n_{\hat{y}}}\cdot B_{\hat{Y}}(\kappa^{n_{\hat{y}}}\cdot \hat{y}, \epsilon_{\hat{y}}) \mid \hat{y} \in A\,\}$, denoted by $\{\,\kappa^{-n_{\hat{y}_i}}\cdot B_{\hat{Y}}(\kappa^{n_{\hat{y}_i}}\cdot \hat{y}_i, \epsilon_{\hat{y}_i}) \mid \hat{y}_i \in A, i=1,2,\dots, m\,\}$. Let $M$ be the maximum of $\{\,n_{\hat{y}_i}\mid i = 1,2,\dots, m\,\}$. Take a non-zero element $\hat{y}$ from $P({\rm cl}_{T^{\max}(\hat{X})}[I(N')])$. 
Then, there exists an integer $z$ satisfying $\nu(\kappa^z) \leq \|\hat{y}\|_{\hat{Y}} < \nu(\kappa^{z+1})$. Because $\kappa^{-z} \cdot \hat{y}$ is in $A$ by the definition of $A$, we can take $i$ so that $\kappa^{n_{\hat{y}_i}-z} \cdot \hat{y} \in B_{\hat{Y}}(\kappa^{n_{\hat{y}_i}} \cdot \hat{y}_i, \epsilon_{\hat{y}_i})$. Then, $\kappa^{n_{\hat{y}_i}-z} \cdot \hat{y}$ does not belong to $P({\rm cl}_{T^{\max}(\hat{X})}[I(N')])$ since $B_{\hat{Y}}(\kappa^{n_{\hat{y}_i}} \cdot \hat{y}_i, \epsilon_{\hat{y}_i})$ does not intersect ${\rm cl}_{T^{\max}(\hat{Y})}[P(I(N'))]$. Thus, we have $\nu(\kappa^{n_{\hat{y}_i}-z}) > 1$ since $N'$ is a balanced neighborhood. Therefore, we deduce that $z \leq n_{\hat{y}_i} \leq M$ and $\|\hat{y}\|_{\hat{Y}} \leq \nu(\kappa^{M+1})$.\par
Now, fix an arbitrary zero's open neighborhood $V$ in $\hat{G} \circ \hat{F}(T)$. Then, $V$ is represented as $V= I^{-1}(W +\hat{F}(T))$, where $W$ is a zero's open neighborhood with respect to $T_{\hat{K}}^{\max}(\hat{X})$. Because $P(W)$ is a  zero's open neighborhood in $\hat{Y}$, the inclusion $B_{\hat{Y}}(0,\delta) \subset P(W)$ holds for some positive $\delta$. We can take $\lambda \in K\setminus \{0\}$ so that $\nu(\lambda) <\frac{\delta}{\nu(\kappa^{M+1})}$ holds. We show that a neighborhood of zero $\lambda \cdot N'$ in $(X,T)$ is contained in $V$.
Let $x$ be an element of $N'$. Since the norm of $P\circ I(x)$ is bounded by $\nu(\kappa^{M+1})$, we have $\|P\circ I(\lambda \cdot x)\|_{\hat{Y}} \leq \nu(\lambda)\nu(\kappa^{M+1}) <\delta$. Thus, $P \circ I(\lambda \cdot x)$ is in $P(W)$, and $I(\lambda \cdot x) \in W + \hat{F}(T)$. By taking the inverse image of $I$, we deduce $\lambda \cdot x$ is in $V$. Thus, we have $\lambda \cdot N' \subset V \in \hat{G} \circ \hat{F}(T)$ and $\hat{G} \circ \hat{F}(T) \subset T$. 
Therefore, $\hat{F}$ and $\hat{G}$ are bijections and preserve the inclusion orders of $(\tau_K(X), \subset)$ and $(\sigma_{\hat{K}}(\hat{X}), \supset)$. Since a bijection which preserve orders of between lattices is a lattice isomorphism, 
we conclude that $(\tau_K(X), \subset)$ is a lattice, which is isomorphic to $(\sigma_{\hat{K}}(\hat{X}), \supset)$ when $\hat{K}$ is a locally compact space.
\end{proof}

%%%%%%%%%%%%%%%%%%%%%%%%%%%%%%%%%%%%%%%%%%%%%%%%%%%%%%%%%%%%%%%%%%%%%%%%%%%%%%%%%%%%%%%%%%%%%%%%%%%%%%%%%%%%%%%%%%%%%%%%%%%%%%%%%%%%%%%%%%%%%%%%%%%%%%%%%%%%%%%%%%%%%%%%%%%%%%%%%%%%%%section 5
\section{Application}
By using the main theorem, we can describe topological properties in terms of subspaces.

\begin{prop}\label{continuous}
Let $X$ and $Y$ be a finite-dimensional vector space over a non-trivial valuation field $(K,\nu)$ whose completion is locally compact. For a linear map $L:X\rightarrow Y$ and for compatible topologies $T_X \in \tau_K(X)$ and $T_Y \in \tau_K(Y)$, the map $L:(X,T_X) \rightarrow (Y,T_Y)$ is continuous if and only if the image of $\hat{F}(T_X)$ by $\hat{L}$ is contained in $\hat{F}(T_Y)$, where $\hat{L}:\hat{X} \rightarrow \hat{Y}$ is a $\hat{K}$-linear map such that the following diagram commutes:
\[
  \begin{CD}
     \hat{K}\times X @>{{\rm id}_{\hat{K}} \times L}>> \hat{K} \times Y \\
  @V{\otimes}VV    @V{\otimes}VV \\
     \hat{X}   @>{\hat{L}}>>  \hat{Y}.
  \end{CD}
\]
\end{prop}

\begin{proof}
Let $I_X:X \rightarrow \hat{X}$ and $I_Y:Y \rightarrow \hat{Y}$ be the maps by which $x\in X$ and $y \in Y$ are sent to $1 \otimes x$ and $1\otimes y$, respectively. Since the above diagram commutes, we have $\hat{L} \circ I_X = I_Y \circ L$ holds.\par
First, we assume $L:(X,T_X) \rightarrow (Y,T_Y)$ is a continuous map. By the continuity of $L$, we have 
$$
\hat{F}(T_X)=\bigcap_{0 \in U \in T_X}{\rm cl}_{T^{\max}(\hat{X})}[I_X(U)] \subset \bigcap_{0 \in V \in T_Y}{\rm cl}_{T^{\max}(\hat{X})}[I_X(L^{-1}(V))].
$$ 
For every zero's open neighborhood $V \in T_Y$, we have $I_X(L^{-1}(V)) \subset \hat{L}^{-1}(I_Y(V))$ because the above diagram commutes. By the continuity of $\hat{L}:(\hat{X},T_{\hat{K}}^{\max}(\hat{X})) \rightarrow (\hat{Y},T_{\hat{K}}^{\max}(\hat{Y}))$, we also have 
$$
{\rm cl}_{T^{\max}(\hat{X})}[\hat{L}^{-1}(I_Y(V))] \subset \hat{L}^{-1}({\rm cl}_{T^{\max}(\hat{Y})}[I_Y(V)]).
$$
 Thus, ${\rm cl}_{T^{\max}(\hat{X})}[I_X(L^{-1}(V))]$ is contained in $\hat{L}^{-1}({\rm cl}_{T^{\max}(\hat{Y})}[I_Y(V)])$ and this implies $\hat{F}(T_X) \subset \hat{L}^{-1}(\hat{F}(T_Y))$. Therefore, the continuity of $L$ implies $\hat{L}(\hat{F}(T_X)) \subset \hat{F}(T_Y)$.\par
Next, we assume that $\hat{L}(\hat{F}(T_X))$ is contained in $\hat{F}(T_Y)$. We denote the strip maps of $\hat{X}$ and $\hat{Y}$ by $G_X$ and $G_Y$, respectively. By extending a basis of $\hat{L}(\hat{F}(T_X))$, we have a direct decomposition of $\hat{F}(T_Y)$ into $\hat{L}(\hat{F}(T_X)) \oplus S'$, where $S'$ is a linear subspace. By 2 of Lemma \ref{prop of F and G_0}, every open subset in $G_Y(\hat{F}(T_Y))$ is represented as $V + \hat{F}(T_Y),~V \in T_{\hat{K}}^{\max}(\hat{Y})$. Thus, its inverse image by $\hat{L}$ is represented as $\hat{L}^{-1}(V+S') + \hat{F}(T_X)$ and is in $G_X\circ\hat{F}(T_X)$. Therefore, $\hat{L}:(\hat{X},G_X \circ \hat{F}(T_X)) \rightarrow (\hat{Y},G_Y \circ \hat{F}(T_Y))$ is continuous. 
 Now, for an open subset $U$ in $T_Y$, there is an open subset $\hat{U} \in G_Y \circ \hat{F}(T_Y)$ such that $U = I_Y^{-1}(\hat{U})$ since $T_Y= \hat{G} \circ \hat{F}(T_Y)$ holds from Theorem \ref{main thm}. Since the above diagram  commutes, we have $L^{-1}(U) = (I_Y \circ L)^{-1}(\hat{U}) = I_X^{-1}(\hat{L}^{-1}(\hat{U}))$. By the continuity of $\hat{L}: (\hat{X},G_X \circ \hat{F}(T_X)) \rightarrow (\hat{Y},G_Y \circ \hat{F}(T_Y))$, we have $\hat{L}^{-1}(\hat{U})$ is in $G_X \circ \hat{F}(T_X)$ and thus, $L^{-1}(U) \in \hat{G} \circ \hat{F}(T_X)$ holds. Again, by Theorem \ref{main thm}, the topology $\hat{G} \circ \hat{F} (T_X)$ coincides with $T_X$ and therefore, $L:(X,T_X) \rightarrow (Y,T_Y)$ is continuous.  
\end{proof}

The following proposition states the equivalent condition of a given compatible topology being Hausdorff when coefficient field is possibly not complete, which is an analogy of the statement: For a compatible topology $T$ on $X$ over non-trivial complete valuation field, $T$ is Hausdorff if and only if $F(T)=\{0\}$.  
\begin{prop}\label{equivalent condition of Hausdorff}
Let $T$ be a compatible topology on $X$. Then, $T$ is a Hausdorff topology if and only if $\hat{F}(T) \cap I(X) =\{0\}$.
\end{prop}
\begin{proof}
First, we assume that $T$ is a Hausdorff topology. We show that every non-zero element $I(x),\,x\not=0$ of $I(X)$ does not belong to $\hat{F}(T)$. Since $T$ is a Hausdorff, there exists an open neighborhood $U \in T \subset T^{\max}(X)$ of zero such that $x \not\in{\rm cl}_{T^{\max}(X)}[U]$. Then, $I(x)$ does not belong to ${\rm cl}_{T^{\max}(\hat{X})}[I(U)]$ because $I$ is injective and because we have ${\rm cl}_{T^{\max}(\hat{X})}[I(U)] \cap I(X) = I({\rm cl}_{T^{\max}}[U])$ from Lemma \ref{regular open}. Therefore, $I(x)$ does not belong to $\hat{F}(T)$ by the definition of $\hat{F}(T)$, which implies $\hat{F}(T) \cap I(X)=\{0\}$.\par
Next, we assume that $\hat{F}(T) \cap I(X)=\{0\}$. To prove that $T$ is a Hausdorff topology, it is enough to show that $\{0\} \subset X$ is a closed subset with respect to $T$. Let $x$ be an arbitrary non-zero element of $X$. Our assumption implies that $I(x)$ does not belong to $\hat{F}(T)$. By the definition of $\hat{F}(T)$, there exists an open neighborhood $U \in T$ of zero such that $I(x) \not\in {\rm cl}_{T^{\max}(\hat{X})}[I(U)]$. Then, $x$ does not belong to $U$ from Lemma \ref{regular open} and from injectivity of $I$. The continuity of a map $X \ni x \mapsto -x \in X$ implies that we have neighborhood $V \in T$ of zero such that $-V \subset U$. The open set $x+V$ is an open neighborhood of $x$, to which zero does not belong. Therefore, the set $\{0\}$ is a closed set, which implies $T$ is a Hausdorff topology.
\end{proof}

\begin{exa}
In this example, we see that there are uncountable many Hausdorff compatible topologies on $\mathbb{Q}^n$ for $n\geq 2$.\par
Let $K\coloneqq \mathbb{Q}$ be the field of rational numbers with the ordinary absolute value, whose completion is the field of real numbers $\mathbb{R}$ and $X\coloneqq\mathbb{Q}^n$ be an $n$-dimensional $\mathbb{Q}$-vector space for $n \geq 2$. We identify $\hat{\mathbb{Q}}$ with $\mathbb{R}$ ($\supset \mathbb{Q}$) and $\hat{X}\coloneqq \hat{\mathbb{Q}} \bigotimes_{\mathbb{Q}} \mathbb{Q}^n$ with $\mathbb{R}^n$. Under this identification, we have a correspondence $\hat{F}$ between $\tau_{\mathbb{Q}}(\mathbb{Q}^n)$ and $\sigma_{\mathbb{R}}(\mathbb{R}^n)$ by Theorem \ref{main thm}. Now, we define an action of the projective linear group $PGL_n(\mathbb{R})$ on $\sigma_{\mathbb{R}}(\mathbb{R}^n)$ by
$$
PGL_n(\mathbb{R})\times \sigma_{\mathbb{R}}(\mathbb{R}^n) \ni ([A],S) \mapsto [A]\cdot S\coloneqq A(S) \in \sigma_{\mathbb{R}}(\mathbb{R}^n),
$$ 
where $A$ is a linear map defined by a matrix which is a representative of $[A]$. Through the correspondence, the group also acts on $\tau_{\mathbb{Q}}(\mathbb{Q}^n)$ by
$$
PGL_n(\mathbb{R})\times \tau_{\mathbb{Q}}(\mathbb{Q}^n) \ni([A],T) \mapsto \hat{G}([A]\cdot\hat{F}(T))\in \tau_{\mathbb{Q}}(\mathbb{Q}^n).
$$
For every point $[x_1, x_2, \, \dots \, , x_n] \in P^{n-1}(\mathbb{R})$, we define a compatible topology, denoted by $T_{[x_1, x_2, \, \dots \,  ,x_n]}$, on $\mathbb{Q}^n$ by $\hat{G}(\mathop{\mathrm{span}_{\mathbb{R}}}\{(x_1,x_2,\, \dots \, ,x_n)\})$. Now, we fix a point $[x'_1,x'_2 , \, \dots \, , x'_n]$ in $P^{n-1}(\mathbb{R})$. By Proposition \ref{continuous}, the subset $\tau'$ of $\tau_{\mathbb{Q}}(\mathbb{Q}^n)$ consists of all topologies $T$ such that $(\mathbb{Q}^n,T)$ is isomorphic to $(\mathbb{Q}^n,T_{[x'_1,x'_2 , \, \dots \, , x'_n]})$ as a $\mathbb{Q}$-topological vector space is the $PGL_n(\mathbb{Q})(\subset PGL_n(\mathbb{R}))$-orbit of $T_{[x'_1,x'_2 , \, \dots \, , x'_n]}$, namely 

\begin{align}
\tau'=
\left\{ \, 
T_{[x_1, x_2, \, \dots \, , x_n]} \in \tau_{\mathbb{Q}}(\mathbb{Q}^n)\, \relmiddle|\,
{\scriptsize
\left[
\begin{array}{r}
x_1\\
x_2\\
\vdots \\  
x_n\\
\end{array}
\right]
}
=A
{\scriptsize
\left[
\begin{array}{r}
x'_1\\
x'_2\\
\vdots \\  
x'_n\\
\end{array}
\right]
},
\,A \in PGL_n(\mathbb{Q})
\,\right\}. \nonumber
\end{align}
\par
Next, we fix a basis $\mathcal{B}$ of $\mathbb{Q}$-vector space $\mathbb{R}$ which contains $1$. 
For every element $\beta\not=1$ of $\mathcal{B}$, the number $\beta$ is irrational, and thus a compatible topology $T_{[1,\beta ,\underbrace{0, \, \dots \, , 0}_{n-2\,\text{zeros}}]}$ is a Hausdorff topology from Proposition \ref{equivalent condition of Hausdorff}. For every distinct two elements $\beta_1$ and $\beta_2$ of $\mathcal{B}$, because $1,\beta_1$ and $\beta_2$ are $\mathbb{Q}$-linear independent, the topology $T_{[1,\beta_1 ,0, \, \dots \, , 0]}$ does not belong to the orbit of $T_{[1,\beta_2,0, \, \dots \, , 0]}$ under the action of $PGL_n(\mathbb{Q})$. Thus, the topological vector spaces $(X,T_{[1,\beta_1 ,0, \, \dots \, , 0]})$ and $(X,T_{[1,\beta_2 ,0, \, \dots \, , 0]})$ are not isomorphic. Therefore, there exist uncountable many Hausdorff compatible topologies on $\mathbb{Q}^n$ such that they are not isomorphic each other. 
\end{exa}

\section*{Acknowledgment}
The author would like to thank to Prof. Ken'ichi Ohshika and Prof. Shinpei Baba for helpful discussions and encouragements. This work was supported by the Research Institute for Mathematical
Sciences (RIMS) in Kyoto University.

\end{document}